\documentclass[reqno]{amsart}

\usepackage[utf8]{inputenc}
\usepackage[T1]{fontenc}
\usepackage{lmodern}
\usepackage[top=3cm, bottom=3cm, left=3cm, right=3cm]{geometry}
\usepackage{amsmath, amssymb, amsthm}
\usepackage{hyperref}
\usepackage{xcolor}
\usepackage{cancel}
\usepackage[normalem]{ulem}
\usepackage{enumitem}
\usepackage{fancyhdr}
\usepackage{booktabs}

\newtheorem{theorem}{Theorem}[section]
\newtheorem{definition}[theorem]{Definition}

\newtheorem{proposition}[theorem]{Proposition}
\newtheorem{lemma}[theorem]{Lemma}

\newtheorem{cor}[theorem]{Corollary}
\theoremstyle{remark}
\newtheorem{remark}{Remark}

\newtheorem{example}{Example}

\newcommand{\cmark}{$\checkmark$}
\newcommand{\Dash}{\textemdash}

\begin{document}

\title[Structural and rigidity  properties of Lie skew braces]{Structural and rigidity  properties of Lie skew braces}

\author{Marco Damele}
\address{(Marco Damele) Dipartimento di Matematica \\
         Universit\`a di Cagliari (Italy)}
        \email{m.damele4@studenti.unica.it}

\author{Andrea Loi}
\address{(Andrea Loi) Dipartimento di Matematica \\
         Universit\`a di Cagliari (Italy)}
         \email{loi@unica.it}

\thanks{
The authors are supported by INdAM and  GNSAGA - Gruppo Nazionale per le Strutture Algebriche, Geometriche e le loro Applicazioni, by GOACT - Funded by Fondazione di Sardegna and 
partially funded by PNRR e.INS Ecosystem of Innovation for Next Generation Sardinia (CUP F53C22000430001, codice MUR ECS00000038) and by the Italian Ministry of Education, University and Research through the PRIN 2022 project "Developing Kleene Logics and their Applications" (DeKLA), project code: 2022SM4XC8.}


\begin{abstract}
We investigate structural and rigidity properties of \emph{Lie skew braces} (LSBs), objects essentially known in the literature as \emph{post--Lie groups}, obtained by endowing a manifold with two compatible group laws that share the same identity element.  LSBs extend skew left braces, which are central to the study of non-involutive set-theoretic solutions of the Yang--Baxter equation, to the smooth category.
Our first main result (Theorem \ref{mainteor}) shows that, for every connected  LSB $(G,\cdot,\circ)$,  linearity (in the simply-connected case) and solvability carry over from $(G,\cdot)$ to $(G,\circ)$, whereas the converse direction is rigid: if $(G,\circ)$ is nilpotent (respectively, semisimple) then $(G,\cdot)$ is forced to be solvable (respectively, isomorphic to $(G,\circ)$). 
Theorem \ref{mainteor2} provides two ``flexibility'' statements: 
every non-linear simply connected Lie group \((G, \cdot )\) admits an LSB \((G,\cdot,\circ)\)  such that  \((G,\circ)\) is linear, 
and every simply connected solvable Lie group \((G, \circ )\) supports {an LSB} \((G,\cdot,\circ)\) such that \((G,\cdot)\) is nilpotent. A third result (Theorem \ref{mainteor3}) provides a complete existence table for non-trivial LSBs across the six standard Lie-group classes, abelian, nilpotent (non-abelian), solvable (non-nilpotent), simple, semisimple (non-simple) and mixed type, identifying precisely when {an LSB} can be built and when only the trivial or no structure occurs.
Both the explicit constructions and the properties established in our theorems rely on a factorisation technique for Lie groups, on the correspondence between LSBs and regular subgroups of the affine group $\operatorname{Aff}(G,\cdot)$, which renders LSB theory equivalent to simply transitive affine actions, and on the theory of post–Lie algebras together with their integrability properties.

\end{abstract}

\maketitle

\tableofcontents  
\section{Introduction}
A \textit{Lie skew brace} (LSB) is a triple \((G, \cdot, \circ)\), where \((G, \cdot)\) and \((G, \circ)\) are real Lie groups sharing the same smooth manifold structure, such that for all \(a, b, c \in G\), the following compatibility condition holds:
\[
a \circ (b\cdot c) = (a \circ b)\cdot a^{-1}\cdot (a \circ c),
\]
where \(a^{-1}\) denotes the inverse of \(a\) in the group \((G, \cdot)\). 
LSBs are essentially just Post-Lie groups introduced in  \cite[Sections 3.3 and 4 ]{Bai2023}. 
More precisely, the categories of LSBs and Post-Lie groups are equivalent (see Subsection \ref{postandLSB} below). 
{an LSB} can be seen as a generalization of the \textit{skew left braces} introduced by Guarnieri and Vendramin in \cite{Guarnieri2017} (see also \cite{Smoktunowicz2018}), originally developed to study non-involutive solutions of the Yang--Baxter equation 
 \cite{Yang2019}),
 and have since become a highly active area of research.
 Indeed, skew left braces correspond to LSB when \(\dim(G) = 0\) and $G$ is equipped with the discrete topology.
\vskip 0.3cm

In this paper, we investigate the structural and rigidity properties of LSBs, with particular emphasis on linearity, solvability, and semisimplicity.

Recall that a Lie group \( G \) is said to be \emph{linear} if there exists a faithful linear representation
\[
\rho : G \rightarrow \mathrm{GL}(V),
\]
where \( V \) is a finite-dimensional real vector space.

The notions of solvability and semisimplicity are of fundamental importance in the study of Lie groups, particularly because of the classical \emph{Levi--Malcev decomposition}, which asserts that any  connected and simply-connected  Lie group \( (G, \cdot) \)  can be expressed as a semidirect product
$
G = R \rtimes S$,
where \( R \) is a maximal  solvable Lie group, called the \emph{radical} of \( G \), and \( S \) is a semisimple Lie subgroup, referred to as a \emph{Levi subgroup}.
This structural decomposition highlights solvability and semisimplicity as the key features that one naturally seeks to understand when analyzing the global structure of Lie groups. 


Our goal is to understand how these structural properties behave under the transition from the Lie group \( (G, \cdot) \) to the Lie group \( (G, \circ) \), and viceversa.

These phenomena are expressed in the following theorem representing our first  main result:

\begin{theorem}\label{mainteor}
Let \( (G, \cdot, \circ) \) be a connected LSB. Then the following hold:

\begin{itemize}
    \item[\textbf{(S1)}] If \( (G, \cdot) \) is linear and simply connected, then \( (G, \circ) \) is linear.
    \item[\textbf{(S2)}] If \( (G, \cdot) \) is solvable, then \( (G, \circ) \) is solvable.
    \item[\textbf{(R1)}] If \( (G, \circ) \) is nilpotent, then \( (G, \cdot) \) is solvable.
   \item[\textbf{(R2)}] If \((G,\circ)\) is semisimple, then \((G,\cdot)\) and \((G,\circ)\) are locally isomorphic.
\end{itemize}
\end{theorem}

Results \textbf{(S1)} and \textbf{(S2)} show that linearity (in the simply connected case) and solvability are preserved when passing from the Lie group \( (G, \cdot) \) to \( (G, \circ) \), whereas \textbf{(R1)} and \textbf{(R2)} highlight rigidity phenomena that constrain the possible structure of \( (G, \cdot) \), given structural assumptions on \( (G, \circ) \).

It is important to observe that the assumption of simple connectedness is essential for \textbf{(S1)} in Theorem~\ref{mainteor} to hold (see Remark~\ref{rmk1mainteor} below). Similarly, without the assumption of connectedness, \textbf{(S2)} fails. Indeed, there exist LSB structures \( (G, \cdot, \circ) \) with \( G \) disconnected such that \( (G, \cdot) \) is solvable, but \( (G, \circ) \) is not (see, for example, \cite[Ex.~3.2]{Nasybullov2019}). In these examples, the group \( G \) has countably many connected components.
In the case where the LSB \( (G, \cdot, \circ) \) has \( G \) finite, i.e., \( \dim G = 0 \) and \( G \) consists of finitely many connected components,  and hence we are in the realm  of left skew braces, the question of whether  the solvability of \( (G, \cdot) \) implies the solvability of \( (G, \circ) \) remains open. This conjecture was first posed by Leandro Vendramin and Agata Smoktunowicz in \cite{Smoktunowicz2018}. Partial progress on this problem has been made; see, for instance, \cite{Nasybullov2019},  \cite{Byott2024}, and \cite{Tsang2019}.

In \cite{Smoktunowicz2018} the authors show that if $(G,\cdot,\circ)$ is a \emph{finite} left skew brace whose {additive} group $(G,\cdot)$ is nilpotent, then the  {multiplicative} group $(G,\circ)$ is necessarily solvable.
Combining this fact with property \textbf{(S2)} of Theorem \ref{mainteor} yields the following compact-group analogue
(see Section \ref{proofs} for a proof and Remark \ref{rmkconj}).
\begin{cor}\label{maincor}
Let $(G,\cdot,\circ)$ be a compact (not necessarily connected) LSB.
If  $(G,\cdot)$ is nilpotent, then  $(G,\circ)$ is solvable. 
\end{cor}

 In contrast to Theorem~\ref{mainteor}, which highlights structural and rigidity phenomena, our second main result demonstrates a form of \emph{flexibility} inherent to LSBs:

\enlargethispage{-\baselineskip} 
\begin{theorem}\label{mainteor2}
\leavevmode

\begin{itemize}
    \item[\textbf{(F1)}] For any non-linear connected and simply-connected Lie group \( (G, \cdot) \), there exists {an LSB} \( (G, \cdot, \circ) \) such that \( (G, \circ) \) is linear.

    \item[\textbf{(F2)}] For any connected and simply-connected solvable Lie group \( (G, \circ) \), there exists {an LSB} \( (G, \cdot, \circ) \) such that \( (G, \cdot) \) is nilpotent.
\end{itemize}
\end{theorem}

It is worth pointing out that to the best of the author's knowledge, the question of whether every \emph{finite solvable group} \((G, \circ)\) admits a \emph{left skew brace structure} \((G, \cdot, \circ)\), such that \((G, \cdot)\) is \emph{nilpotent}, remains open.

\vskip 0.3cm

The proofs of  $\textbf{(R1)}$ and $\textbf{(R2)}$ of Theorem \ref{mainteor} are based on associating to {an LSB} 
\( (G, \cdot, \circ) \) a post Lie algebra structure, namely a triple \((\mathfrak{g}, [\cdot,\cdot]_{\cdot}, [\cdot,\cdot]_{\circ})\),
where  \((\mathfrak{g}, [\cdot,\cdot]_{\cdot})\) (resp. \((\mathfrak{g}, [\cdot,\cdot]_{\circ})\)) is the Lie algebra associated to \( (G, \cdot) \) (resp. \( (G, \circ) \)) related by a post-Lie product (cf. Def. \ref{postLieAlgebras} below).

 Not every post-Lie algebra arises as the infinitesimal counterpart of {an LSB} structure.
Determining which PLAs can be integrated, i.e., obtained by differentiating {an LSB}, is a subtle and active area of research (see \cite{Bai2023, Burde2018postliealgebranilpotent, Burde2021crystallographicpostliealgebra, Burde2022rigiditypostliealgebra, Ebner2019postliealgebra} and references therein).
It is therefore of particular interest to construct non-trivial examples of LSBs. 


 
 
 
  
   
   
     
     








In the following theorem, our third main result,  we summarize the existence of non-trivial LSBs structures for various combinations of \((G, \cdot, \circ)\), as illustrated in the table. A check mark (\cmark) indicates that a non-trivial LSB structure of the corresponding type exists, while a  dash ($-$) that no such example is possible  and  a $\cong$ that the LSB is trivial, i.e.  \((G, \cdot)\cong (G, \circ)\). We write \(\cong_{\mathrm{loc}}\) to denote local isomorphism of Lie groups, i.e. isomorphism of the corresponding Lie algebras.
Notice also, that to avoid overlaps between classes, we adopt the following conventions:
\begin{itemize}
  \item \textbf{ab} refers to an abelian Lie group;
  \item \textbf{nil} denotes a non-abelian nilpotent Lie group;
  \item \textbf{solv} denotes a solvable but non-nilpotent Lie group;
  \item \textbf{simp} denotes a simple Lie group 
\item \textbf{ssimp} denotes a semisimple but non-simple Lie group;
\item \textbf{mixed-type} refers to a Lie group that does not fall into any of the categories above.
\end{itemize}

\begin{theorem}\label{mainteor3}
The existence of non-trivial LSB  $(G, \cdot, \circ)$ for different types of Lie groups \((G, \cdot )\) and \((G, \circ)\)  is summarized in the following table:

\begin{table}[h]\label{tableLSB}
  \centering
  \small 
  \renewcommand{\arraystretch}{1.1}
  \setlength{\tabcolsep}{6pt}

  \begin{tabular}{|l*{6}{c}|} 
    \hline
    & $(G,\cdot)$ \textbf{ab} 
    & $(G,\cdot)$ \textbf{nil} 
    & $(G,\cdot)$ \textbf{solv} 
    & $(G,\cdot)$ \textbf{simp} 
    & $(G,\cdot)$ \textbf{ssimp} 
    & $(G,\cdot)$ \textbf{mixed-type} \\
    \hline
    $(G,\circ)$ \textbf{ab}            & $\cong$ & \cmark & \cmark & \Dash & \Dash & \Dash \\
    $(G,\circ)$ \textbf{nil}           & \cmark & \cmark & \cmark & \Dash & \Dash & \Dash \\
    $(G,\circ)$ \textbf{solv}          & \cmark & \cmark & \cmark & \cmark & \cmark & \cmark \\
    $(G,\circ)$ \textbf{simp}          & \Dash  & \Dash & \Dash & $\cong$ & {\Dash} & \Dash \\
    $(G,\circ)$ \textbf{ssimp}         & \Dash & \Dash & \Dash & \Dash & {$\cong_{loc}$} & \Dash \\
    $(G,\circ)$ \textbf{mixed-type}    & \Dash & \Dash & \Dash & \cmark & \cmark & \cmark \\
    \hline
  \end{tabular}
\end{table}
\end{theorem}

To conclude, we stress how our results complement and extend, in the solvable and semisimple cases, the classification provided in \cite[Th.~5.1]{Burde2024postliealgebraperfect}, which concerns PLAs.
This comparison allows for the immediate identification of several PLAs that cannot be integrated into {an LSB}. For example, there exist PLAs 
\((\mathfrak{g}, [\cdot,\cdot]_{\cdot}, [\cdot,\cdot]_{\circ})\) 
where 
\((\mathfrak{g}, [\cdot,\cdot]_{\cdot})\) 
is a non-nilpotent solvable Lie algebra, while 
\((\mathfrak{g}, [\cdot,\cdot]_{\circ})\) 
is reductive but non-abelian (see Example~\ref{exnonint} below). According to condition \textbf{(S2)} in Theorem~\ref{mainteor}, such PLAs are not integrable into an LSB.

Another interesting implication of our results concerns the following open question posed by the authors in \cite{Burde2024postliealgebraperfect}, which relates to the existence of PLAs 
\((\mathfrak{g}, [\cdot,\cdot]_{\cdot}, [\cdot,\cdot]_{\circ})\) 
such that 
\((\mathfrak{g}, [\cdot,\cdot]_{\cdot})\) 
is nilpotent and 
\((\mathfrak{g}, [\cdot,\cdot]_{\circ})\) 
is perfect, i.e.  $\mathfrak{g}= [\mathfrak{g},\mathfrak{g}]_{\circ}$. While we are currently not able to answer this question definitively, we can nevertheless state that any such PLA, if it exists, would not be integrable into an LSB due to condition \textbf{(S2)} in Theorem~\ref{mainteor}.

A more comprehensive comparison between our findings and those of \cite{Burde2024postliealgebraperfect} will be presented in a forthcoming work.

The paper is organised as follows.  
Section \ref{secLSB} recalls the necessary background on Lie skew braces, post–Lie groups and algebras, and their realisation as regular subgroups of the affine group $\operatorname{Aff}(G,\cdot)$; in particular, it describes in detail the correspondence with simply transitive affine actions, introduces the $\lambda$-action, and sets up the {\emph{derivation functor}} that associates to each LSB its corresponding post Lie algebra.  
Section \ref{proofs} contains the proofs of  Theorem \ref{mainteor}-\ref{mainteor2} and 
Corollary \ref{maincor}: 
parts $\bf(S1)$-$\bf(S2)$ are obtained by lifting LSBs to the universal cover and exploiting the linearity of $\operatorname{Aff}_0(G,\cdot)$, while the rigidity statements $\bf(R1)$-$\bf(R2)$ 
follow from corresponding rigidity theorems for post-Lie algebras.  
The flexibility results and the existence table are derived through an explicit factorisation technique for Lie groups, the equivalence between LSBs and simply transitive affine actions (cf. Theorem \ref{propLSBaff} below), and integrability criteria for post-Lie {algebras.}

 \section{Lie skew braces, post-Lie groups, post-Lie algebras and affine actions}\label{secLSB}
\subsection{Lie skew braces}
Let \( (G, \cdot, \circ) \) be a Lie skew brace (LSB), meaning that \( (G, \cdot) \) and \( (G, \circ) \) are real Lie groups sharing the same  smooth manifold structure, and such that the identity  
\begin{equation} \label{eqLSB}
a \circ (b \cdot c) = (a \circ b) \cdot a^{-1} \cdot (a \circ c)
\end{equation}
holds for all \( a, b, c \in G \), where \(a^{-1}\) denotes the inverse of \(a\) in the group \((G, \cdot)\). 
From this it follows that the two groups also {share the same identity} element denoted here by $e$.
One can then define the category \( \mathbf{LSB} \), whose \emph{objects} are LSBs, and whose \emph{morphisms} between two LSBs  
\((G_1, \cdot_1, \circ_1)\) and \((G_2, \cdot_2, \circ_2)\) are \emph{LSB homomorphisms}, that is, smooth maps \( \varphi: G_1 \to G_2 \) such that
\[
\varphi(a \cdot_1 b) = \varphi(a) \cdot_2 \varphi(b) \quad \text{and} \quad \varphi(a \circ_1 b) = \varphi(a) \circ_2 \varphi(b)
\]
for all \( a, b \in G_1 \).

The category of  \( \mathbf{LSB} \)  forms a \emph{subcategory} of 
\emph{left skew braces}, as defined in \cite{Guarnieri2017} (in the purely algebraic sense, without requiring a Lie group structure), namely the category consisting of those objects and morphisms where the group operations are considered on underlying sets rather than smooth manifolds.
Indeed, all the results described in this subsection are natural extensions to the smooth case of those established for left skew braces in the algebraic setting, as presented in \cite{Guarnieri2017}.

\begin{example}\label{exLSBtriv} (cf. \cite[Ex. 1.3 ]{Guarnieri2017})
Let \( (G, \cdot) \) be a real Lie group. Define \( a \circ b = a\cdot b \) for all \( a, b \in G \). Then \( (G, \cdot, \circ) \) is {an LSB}. Similarly, defining \( a \circ b = b\cdot  a\) also yields {an LSB} structure on \( G \).
These LSBs are isomorphic if and only if $(G, \cdot)$ is abelian. For our aim these LSBs
are not interesting since in both cases  $(G, \cdot)\cong (G, \circ)$.
\end{example}

\begin{example} \label{exLSBprod}
Let \( (G_1, \cdot_1, \circ_1) \) and \( (G_2, \cdot_2, \circ_2) \) be two LSBs. Then the triple
\[
(G_1 \times G_2, \cdot, \circ)
\]
is also {an LSB}, where the operations are defined componentwise by
\[
(a_1, b_1) \cdot (a_2, b_2) := (a_1 \cdot_1 a_2,\; b_1 \cdot_2 b_2), \quad
(a_1, b_1) \circ (a_2, b_2) := (a_1 \circ_1 a_2,\; b_1 \circ_2 b_2),
\]
for all \( a_1, a_2 \in G_1 \) and \( b_1, b_2 \in G_2 \).
This LSB is referred to as the \emph{direct product} of the LSBs \( (G_1, \cdot_1, \circ_1) \) and \( (G_2, \cdot_2, \circ_2) \).
\end{example}

\begin{example} \label{exLSBzappa} (cf.\ \cite[Ex.~1.6]{Guarnieri2017}) 
Let \( G \) be a real Lie group admitting a factorization \( G = G_1 G_2 \), where \( G_1 \) and \( G_2 \) are Lie subgroups such that \( G_1 \cap G_2 = \{1\} \). 
 In particular, every element \( a \in G \) can be uniquely written as \( a = a_1\cdot a_2 \) for some \( a_1 \in G_1 \), \( a_2 \in G_2 \).
This factorization allows us to define a new group operation \( \circ \) on the underlying manifold of \( G \), as follows: for any \( a, b \in G \), let \( a = a_1\cdot a_2 \), with \( a_1 \in G_1 \) and  \( a_2 \in G_2 \)  . Then define
\[
a \circ b := a_1 \cdot b \cdot a_2.
\]
With this operation, the triple \( (G, \cdot, \circ) \) becomes {an LSB}. 
Indeed the map
\[
\Phi: G_1 \times G_2 \longrightarrow G, \quad (a_1,a_2) \mapsto a_1 \cdot a_2^{-1}
\]
is a Lie group isomorphism between $(G, \circ )$ and the direct product $G_1\times G_2$.
\end{example}

\vskip 0.3cm

Let \((G, \cdot) \) be a connected real Lie group. The group of Lie group automorphisms of \( G \), denoted \( \mathrm{Aut}(G, \cdot) \), consists of all Lie group isomorphisms of \( G \), i.e.  smooth group homomorphisms from \( G \) to itself with smooth inverses. When endowed with the subspace topology induced by the inclusion
\[
\mathrm{Aut}(G, \cdot) \subset C^0(G, G),
\]
where \( C^0(G, G) \) is the space of continuous maps from \( G \) to itself equipped with the compact-open topology, the group \( \mathrm{Aut}(G, \cdot) \) becomes a real Lie group of finite dimension. For further details, see \cite{Hochschild1952aut} and \cite{Dani1992}.

\begin{example} \label{exLSBrevzappa}
Let \( G_1 \) and \( G_2 \) be connected real Lie groups, and let 
\[
\alpha : G_2 \to \mathrm{Aut}(G_1)
\]
be a Lie group homomorphism. Then the set \( G_1 \times G_2 \), endowed with the usual componentwise group operation
\[
(a_1, b_1) \cdot (a_2, b_2) := (a_1 \cdot a_2,\, b_1 \cdot b_2),
\]
and with the twisted operation
\[
(a_1, b_1) \circ_{\alpha} (a_2, b_2) := \left( a_1 \cdot \alpha(b_1)(a_2),\; b_1 \cdot b_2 \right),
\]
forms {an LSB} \( (G_1 \times G_2, \cdot, \circ_{\alpha}) \).
Here, \( (G_1 \times G_2, \circ_{\alpha}) \) is precisely the {\em semidirect product} Lie group \( G_1 \rtimes_{\alpha} G_2 \). 

\end{example}

\vskip 0.3cm

\begin{definition}
The  {\em affine group}\footnote{Some authors, especially in the context of finite groups, but also occasionally in the Lie group setting (see e.g., \cite{Wu1986}),  refer to $\mathrm{Aff}(G)$
as the  \emph{holomorph} of $G$,  and denote it by \( \mathrm{Hol}(G) \) instead of \( \mathrm{Aff}(G) \).}
associated to a connected Lie group  \( (G, \cdot) \) is the Lie group given by the semidirect product
\[
\mathrm{Aff}(G, \cdot) = (G, \cdot) \rtimes \mathrm{Aut}(G, \cdot),
\]
where  \( \mathrm{Aut}(G, \cdot) \)  acts on \( G\) by evaluation. 
\end{definition}

The group operation on \( \mathrm{Aff}(G, \cdot) \) is then given by
\[
(a, f) (b, g) = (a \cdot f(b), f  g),
\]
for all \( a, b \in G \), \( f, g \in \mathrm{Aut}(G, \cdot) \).
Any {Lie} subgroup $H$ of $\mathrm{Aff}(G, \cdot)$ acts smoothly  on $G$:
$$(x, f)a=x\cdot f(a),\  a\in G, (x, f)\in H.$$

In particular $\mathrm{Aff}(G, \cdot)$ acts transitively on $G$ and the stabilizer 
of any $a\in G$ is a Lie subgroup of  $\mathrm{Aff}(G, \cdot)$ isomorphic (as Lie group)
 to {$\mathrm{Aut}(G, \cdot)$}.

\begin{definition}\label{regsubgroup}
Let $(G, \cdot)$ be a connected Lie group.
A subgroup \( H \subseteq \mathrm{Aff}(G, \cdot) \) is said to be \emph{regular}
if for each  $a\in G$ there exists a unique $(x, f)\in H$  such that $x\cdot f(a)=e$,  where $e$ denote the identity element of $(G, \cdot)$.
\end{definition}

If $H$ is a regular subgroup of  $\mathrm{Aff}(G, \cdot)$ then it is not hard to see that   the projection  \( \pi_1 : \mathrm{Aff}(G) \to G \) onto the first factor, when restricted to \( H \), is a bijection (see \cite[Lemma 4.1]{Guarnieri2017}).

\begin{definition}\label{regsubgroup}
Let $(G, \cdot)$ be a connected Lie group.
A  subgroup \( H \subseteq \mathrm{Aff}(G, \cdot) \) is said to be \emph{ Lie regular} if the following conditions are satisfied:
\begin{enumerate}
\item
$H$ is a regular subgroup of $ \mathrm{Aff}(G, \cdot)$;
\item
$H$ is  closed as a subspace of $\mathrm{Aff}(G, \cdot)$.
\end{enumerate}
\end{definition}

\begin{remark}\rm\label{rmkclosed}
By the Closed Subgroup Theorem, assumption (2) implies that \(H\) is an embedded Lie subgroup of $\mathrm{Aff}(G, \cdot)$.
In fact, one can weaken (2) to: \lq\lq\(H\) is an \emph{immersed} Lie subgroup of $\mathrm{Aff}(G, \cdot)$'.' Indeed, the action of \(H\) on \(G\),
\[
H\times G\longrightarrow G,\qquad (x,f)\cdot a:=x\,f(a),
\]
is smooth, free, and transitive. For a smooth transitive action of a Lie group \(H\) on a Hausdorff \emph{smooth manifold} \(X\), the canonical map
\[
\Phi:\; H/H_x \longrightarrow X,\qquad \Phi(hH_x)=h\cdot x,
\]
is a homeomorphism (in fact, a diffeomorphism once \(H/H_x\) is given its standard smooth structure, see e.g. \cite[Th. 21.18]{Lee-ISM}).
Applying this with \(X=G\) and \(x=e\), the stabilizer is \(H_e=\{e\}\), hence
\[
\pi_1|_H:\; H \longrightarrow G
\]
is a homeomorphism, where \(\pi_1:G\times \mathrm{Aut}(G, \cdot)\to G\) denotes the first projection under the identification \(\mathrm{Aff}(G, \cdot)\cong G\rtimes \mathrm{Aut}(G, \cdot)\).
Write \((\pi_1|_H)^{-1}(a)=(a,\varphi_a)\) and define
\[
\varphi:G\longrightarrow \mathrm{Aut}(G, \cdot),\qquad \varphi(a):=\varphi_a.
\]
Then \(\varphi=\pi_2\circ(\pi_1|_H)^{-1}\) is continuous (here \(\pi_2\) is the second projection), and
\[
H=\{\, (a,\varphi(a)) : a\in G \,\}
\]
is the graph of a continuous map into the Hausdorff group \(\mathrm{Aut}(G, \cdot)\). Therefore \(H\) is closed in
\(\mathrm{Aff}(G, \cdot)\).
\end{remark}

\vskip 0.3cm

Now, let us fix {an LSB} \( (G, \cdot, \circ) \). Define the  \emph{lambda-action}
\begin{equation}\label{lambda}
\lambda: (G, \circ) \to \mathrm{Aut}(G, \cdot), \quad a \mapsto \lambda_a,
\end{equation}
where
\begin{equation}\label{lambdaspec}
\lambda_a(b) = a^{-1}\cdot (a \circ b), \quad \forall b \in G.
\end{equation}
It is easily seen (see, e.g.,  \cite[Corollary 1.10]{Guarnieri2017}, that $\lambda$ is a group homomorphism. This action is particularly important, as it expresses the group operations \(\circ\) in terms of \(\cdot\), and viceversa:
$$a \circ b = a\cdot \lambda_a(b), \quad a\cdot b = a \circ \lambda_{\overline{a}}(b), \forall a, b \in G,$$
where \(\overline{a}\) denotes the inverse of \(a\) in the group \((G, \circ)\)

The following lemma shows that $\lambda$  is indeed a Lie group homomorphism.

\begin{lemma} \label{lambdasmooth}
Let \( (G, \cdot, \circ) \) be {an LSB}.
Then the map \( \lambda \) is smooth.
\end{lemma}

\begin{proof}
Since both \( (G, \cdot) \) and \( \mathrm{Aut}(G, \cdot) \) are Lie groups, and \( \lambda \) is a group homomorphism, it suffices to prove that \( \lambda \) is continuous.
Observe that, by \eqref{lambdaspec} and the fact that the group operations on \( (G, \cdot) \) and \( \mathrm{Aut}(G, \cdot) \) are continuous (in fact, smooth), the map
\[
\tilde{\lambda} \colon G \times G \to G, \quad (a, b) \mapsto \lambda_a(b)
\]
is continuous.

Now, since any Lie group is a locally compact topological space, it follows from \cite[Theorem 46.11]{Munkres2000topology} that for each \( a \in G \), the map
\[
\lambda \colon G \to  \mathrm{Aut}(G, \cdot)\subset C^0(G, G), \quad a \mapsto \lambda_a
\]
is continuous when \( \mathrm{Aut}(G, \cdot) \) is viewed as a subspace of \( C^0(G, G) \) with the compact-open topology. This completes the proof.
\end{proof}

The connection between LSBs and regular subgroups of the affine group is expressed by the following proposition, which generalizes the classical correspondence for skew left braces (see \cite[Th.~4.2 and Prop.~4.3]{Guarnieri2017}). We sketch the proof in the smooth setting.

\begin{proposition}\label{proplambda}
Let \( (G, \cdot) \) be a connected real Lie group. Then there is a one-to-one correspondence between:
\begin{itemize}
    \item LSBs structures \( (G, \cdot, \circ) \), and
    \item regular Lie subgroups \( H \leq \mathrm{Aff}(G, \cdot) \),
\end{itemize}
such that \( (G, \circ) \simeq H \) as Lie groups. Moreover, two LSB structures on \( G \) are isomorphic if and only if the corresponding regular subgroups are conjugate in \( \mathrm{Aff}(G, \cdot) \) via an element of \( \mathrm{Aut}(G, \cdot) \).
\end{proposition}

\begin{proof}
Suppose \( (G, \cdot, \circ) \) is {an LSB}. Define the regular Lie subgroup  \( H = \{ (a, \lambda_a) \mid a \in G \} \leq \mathrm{Aff}(G, \cdot) \), where $\lambda_a$ is given by \eqref{lambdaspec}.  The map
\[
\Psi \colon G \to H, \quad a \mapsto (a, \lambda_a)
\]
is a group homomorphism from \( (G, \circ) \) to \( H \), since
\[
\Psi(a \circ b) = (a \lambda_a(b), \lambda_{a \circ b}) = (a, \lambda_a)(b, \lambda_b) = \Psi(a)\Psi(b).
\]
Smoothness of \( \lambda \) (cf.\ Lemma~\ref{lambdasmooth}) ensures that \( \Psi \) is a diffeomorphism, and hence \( H \) is a regular Lie subgroup of \( \mathrm{Aff}(G, \cdot) \),  isomorphic to \( (G, \circ) \) as a Lie group.

Conversely, let \( H \subseteq \mathrm{Aff}(G, \cdot) \) be a regular Lie subgroup. Then the projection \( \pi_1|_H \colon H \to G \) is a diffeomorphism. Transporting the group law of \( H \) to \( G \) via this diffeomorphism defines a smooth operation \( \circ \), given by
\[
a \circ b := \pi_1\left( (\pi_1|_H)^{-1}(a) \cdot (\pi_1|_H)^{-1}(b) \right)=a\cdot f(b),
\]
where \( (\pi_1|_H)^{-1}(a) = (a, f) \). It follows that \( (G, \circ) \) is a Lie group and (a direct computation shows that) \( (G, \cdot, \circ) \) is {an LSB}.

The final assertion concerning the correspondence between isomorphism classes of LSBs and conjugacy classes of regular subgroups is of algebraic nature and can be found in  \cite[Prop.~4.3]{Guarnieri2017}. The proof is omitted.
\end{proof}

\subsection{Post-Lie groups}\label{postandLSB}
We recall the definition of a  post-Lie group given in \cite{Bai2023}. 

\begin{definition}\label{defpostliegroup}
    A post Lie group is a triple $(G,\cdot,\triangleright)$ where $(G,\cdot)$ is a real  Lie group and $\triangleright: G \times G \rightarrow G$ is a smooth map such that: 
    \begin{enumerate}
        \item 
        For every $a \in G$ the map 
        \begin{equation}\label{Ltriangle}
        L^{\triangleright}_{a}:G \rightarrow  G, b \mapsto  L^{\triangleright}_{a}(b):=a \triangleright b
        \end{equation} 
        is an automorphism of $(G, \cdot)$, i.e. $L^{\triangleright}_{a}(b\cdot c)=(a \triangleright b)\cdot (a \triangleright c)$, for all $a,b,c \in G$;
        \item For all $a,b,c \in G$,
        $$(a \cdot (a \triangleright b)) \triangleright c= a \triangleright ( b \triangleright c).$$
    \end{enumerate}
\end{definition}
Given two post Lie groups $(G_1,\cdot_1,\triangleright_1),(G_2,\cdot_2,\triangleright_2)$ a smooth map $\Phi: G_1 \rightarrow G_2$ is said to be an homomorphism between the post Lie groups $(G_1,\cdot_1,\triangleright_1)$ and $(G_2,\cdot_2,\triangleright_2)$, if for every $a,b \in G_1$ we have $\Phi(a \cdot_1 b)=\Phi(a) \cdot_2 \Phi(b)$ and $\Phi(a \triangleright_1 b)=\Phi(a) \triangleright_2 \Phi(b)$.
Therefore we can speak about the category of post-Lie group $\mathbf{PLG}$, 
whose objects are post-Lie groups and whose morphisms are homomorphisms between the post-Lie groups.
\begin{proposition}
The categories $\mathbf{LSB}$ and $\mathbf{PLG}$ are isomorphic.
\end{proposition}

\begin{proof}
As shown in \cite[Prop. 3.22]{Bai2023}, there is a functor
\[
F: \mathbf{PLG} \rightarrow \mathbf{LSB}, \quad (G,\cdot, \triangleright) \mapsto (G,\cdot,\circ)
\]
where for all \(a,b \in G\) the operation \(\circ\) is defined algebraically by
\[
a \circ b := a \cdot (a \triangleright b).
\]

Conversely, by \cite[Thm 3.24]{Bai2023}, one defines a functor
\[
G: \mathbf{LSB} \rightarrow \mathbf{PLG}, \quad (G,\cdot, \circ) \mapsto (G,\cdot,\triangleright)
\]
where, for all \(a,b \in G\), the post-Lie product is given by
\[
a \triangleright b := \lambda_a(b),
\]
with \(\lambda_a(b)\) defined by \eqref{lambdaspec}.
Since, by Lemma \ref{lambdasmooth}, \(\lambda: G  \to \mathrm{Aut}(G, \cdot)\)  is smooth, this construction is well-defined in the category of smooth manifolds. 
Finally, by \cite[Thm 3.25]{Bai2023}, the two functors \(F\) and \(G\) are mutually inverse. Therefore, the categories \(\mathbf{LSB}\) and \(\mathbf{PLG}\) are isomorphic.
\end{proof}

\subsection{Post-Lie algebras}\label{subsecPLA}
Although the correspondence between Lie groups and Lie algebras is classical and well known, we briefly recall it here to fix notation.
Given any Lie group \( (G, \cdot) \), there is an associated Lie algebra \( \mathrm{Lie}(G, \cdot) = (\mathfrak{g}, [\cdot,\cdot]_{\cdot}) \), defined as the tangent space \( T_eG \) at the identity element \( e \), endowed with the Lie bracket induced by the left-invariant vector fields on \( G \).
Conversely, for every finite-dimensional real Lie algebra \( (\mathfrak{g}, [\cdot,\cdot]) \), there exists a real Lie group \( (G, \cdot) \) whose Lie algebra is \( (\mathfrak{g}, [\cdot,\cdot]_{\cdot}) \). In fact, such a group \( G \) can be obtained as a quotient of the unique simply-connected Lie group integrating \( (\mathfrak{g}, [\cdot,\cdot]) \).

\begin{definition} \label{postLieAlgebras}
    A \emph{post-Lie algebra} (PLA) is a  quadruple \( (\mathfrak{g}, [\cdot,\cdot]_{\cdot}, [\cdot,\cdot]_{\circ},  \triangleright )\) such that both \( (\mathfrak{g}, [\cdot,\cdot]_{\cdot}) \) and \( (\mathfrak{g}, [\cdot,\cdot]_{\circ}) \) are real Lie algebras (on the same finite dimensional real vector space $\mathfrak{g}$), and     \[
    \triangleright: \mathfrak{g} \times \mathfrak{g} \rightarrow \mathfrak{g}, \quad (\zeta,\eta) \mapsto \zeta \triangleright \eta,
    \]
    is a   bilinear map, called the {\em post-Lie product}, 
    satisfying the following identities for all \( \zeta, \eta, \gamma \in \mathfrak{g} \):
    \begin{enumerate}
        \item \( [\zeta,\eta]_{\circ} - [\zeta,\eta]_{\cdot} = \zeta \triangleright \eta - \eta \triangleright \zeta \),
        \item \( \zeta \triangleright [\eta,\gamma]_{\cdot} = [\zeta \triangleright \eta,\gamma]_{\cdot} {+} [\eta, \zeta \triangleright \gamma]_{\cdot} \),
        \item \( [\zeta,\eta]_{\circ} \triangleright \gamma = \zeta \triangleright (\eta \triangleright \gamma) - \eta \triangleright (\zeta \triangleright \gamma) \).
    \end{enumerate}
\end{definition}

\begin{remark}
 Notice that  by (1) in Definition \ref{postLieAlgebras} the Lie algebra structure  $[\cdot,\cdot]_{\circ}$
 is uniquely determined by the pair $([\cdot,\cdot]_{\cdot},  \triangleright)$.
 Hence, one can also define (see, e.g. \cite{Bai2023} and  \cite{Burde2012affineaction}) a PLA as a triple 
 \( (\mathfrak{g}, [\cdot,\cdot]_{\cdot}, \triangleright) \), where \( (\mathfrak{g}, [\cdot,\cdot]_{\cdot}) \) is a real Lie algebra and 
    \[
    \triangleright: \mathfrak{g} \times \mathfrak{g} \to \mathfrak{g}, \quad (\zeta,\eta) \mapsto \zeta \triangleright \eta
    \]
    is a bilinear map satisfying:
    \begin{enumerate}
        \item \( \zeta \triangleright [\eta,\gamma]_{\cdot} = [\zeta \triangleright \eta, \gamma]_{\cdot} {+} [\eta, \zeta \triangleright \gamma]_{\cdot} \),
        \item \( ([\zeta,\eta]_{\cdot} + \zeta \triangleright \eta - \eta \triangleright \zeta) \triangleright \gamma = \zeta \triangleright (\eta \triangleright \gamma) - \eta \triangleright (\zeta \triangleright \gamma) \).
    \end{enumerate}
    Then one defines the so-called \emph{sub-adjacent Lie algebra} \( (\mathfrak{g}, [\cdot,\cdot]_{\circ}) \) via
    \[
    [\zeta,\eta]_{\circ} := [\zeta,\eta]_{\cdot} + \zeta \triangleright \eta - \eta \triangleright \zeta.
    \]
    Indeed, {it} is not hard to  verify that the quadruple  
    \( (\mathfrak{g}, [\cdot,\cdot]_{\cdot}, [\cdot,\cdot]_{\circ},  \triangleright )\) is a  PLA. 
   \end{remark}

An important class of examples of post-Lie algebras is given by {\em pre-Lie algebras}
(also known as  {\em left-symmetric algebra{s}}), which arise when the Lie bracket \( [\cdot,\cdot]_{\cdot} \) is trivial. 
In this setting, the post-Lie product defines a compatible Lie algebra structure via antisymmetrization.

\begin{definition}\label{preliealg} \label{preLieAlgebras}
A \emph{pre-Lie algebra} is a triple  \( (\mathfrak{g}, [\cdot,\cdot]_{\circ} ,  \triangleright)\) consisting of a    finite dimensional Lie algebra
    \( (\mathfrak{g}, [\cdot,\cdot]_{\circ} )\)   
    together with a bilinear map 
     \[
    \triangleright : \mathfrak{g} \times \mathfrak{g} \rightarrow \mathfrak{g}, \quad (\zeta, \xi) \mapsto \zeta \triangleright \xi
    \]
 satisfying the identity
    \begin{equation}\label{idpre}
    \zeta \triangleright (\xi \triangleright \gamma) - \xi \triangleright (\zeta \triangleright \gamma)
    = (\zeta \triangleright \xi) \triangleright \gamma - (\xi \triangleright \zeta) \triangleright \gamma
    \quad \text{for all } \zeta, \xi, \gamma \in \mathfrak{g}.
    \end{equation}
   \end{definition}

The \emph{category of post-Lie algebras}, denoted by \( \mathbf{PLA} \), is the category whose objects are PLAs  \( (\mathfrak{g}, [\cdot,\cdot]_{\cdot},  [\cdot,\cdot]_{\circ} )\) and morphisms between two PLAs \( (\mathfrak{g}_1, [\cdot,\cdot]_{\cdot_1}, [\cdot,\cdot]_{\circ_1},  \triangleright_{1} )\) and \ \( (\mathfrak{g}_2, [\cdot,\cdot]_{\cdot_2}, [\cdot,\cdot]_{\circ_2},  \triangleright_{2}) \) are linear maps 
 $\phi: \mathfrak{g}_1 \to \mathfrak{g}_2$ such that for all \( \xi, \eta \in \mathfrak{g}_1 \),
    \[
    \phi([\xi,\eta]_{\cdot_1}) 
    = [\phi(\xi), \phi(\eta)]_{\cdot_2} \quad \text{and} \quad 
     \phi (\xi \triangleright_{1} \eta)= 
     \phi(\xi) \triangleright_{2} \phi(\eta).
    \]
    (note that  also $\phi([\xi,\eta]_{\circ_1}) 
    = [\phi(\xi), \phi(\eta)]_{\circ_2}$ holds true  by (1) in Definition \ref{postLieAlgebras}).

PLAs and pre-Lie algebras have attracted significant attention in recent years since  they  arise in a variety of contexts including affine manifolds, Lie group actions, crystallographic groups, affine representations of Lie algebras, quantum field theory, operad theory, Rota--Baxter operators, and deformation theory. For further details on pre- and post-Lie algebras, we refer the reader to \cite{Burde2012affineaction, Burde2013postliealgebra, Burde2016postliealgebrapair, Burde2018postliealgebranilpotent, Burde2019rotabaxterpostliealgebra}, and the references therein. A comprehensive survey is also provided in \cite{Burde2021crystallographicpostliealgebra}.

\vskip 0.3cm
Given {an LSB} \( (G, \cdot, \circ) \), one can associate two Lie algebras:
\[
\mathrm{Lie}(G, \cdot) := (\mathfrak{g}, [\cdot,\cdot]_{\cdot}), \qquad 
\mathrm{Lie}(G, \circ) := (\mathfrak{g}, [\cdot,\cdot]_{\circ}),
\]
and define a bilinear map \( \triangleright: \mathfrak{g} \times \mathfrak{g} \to \mathfrak{g} \) by
\begin{equation}\label{trianglelambda}
\xi \triangleright \eta := \lambda_{*e}(\xi)(\eta), \qquad \text{for all } \xi, \eta \in \mathfrak{g},
\end{equation}
where \( \lambda_{*e} \) is the differential at the identity \( e \) of the map
\[
\lambda: (G, \circ) \to \mathrm{Aut}(G, \cdot),
\]
defined by \eqref{lambda}. In particular,
\[
\lambda_{*e}: (\mathfrak{g}, [\cdot,\cdot]_{\circ}) \longrightarrow \mathrm{Der}\left(\mathfrak{g}, [\cdot,\cdot]_{\cdot}\right),
\]
where  \( \mathrm{Der}(\mathfrak{g}, [\cdot,\cdot]_{\cdot}) \) denotes the Lie algebra of derivations of \( (\mathfrak{g}, [\cdot,\cdot]_{\cdot}) \), that is, the Lie algebra of the automorphism group \( \mathrm{Aut}(G, \cdot) \).
As shown in \cite{Bai2023}, the quadruple
$(\mathfrak{g}, [\cdot,\cdot]_{\cdot}, [\cdot,\cdot]_{\circ}, \triangleright)$
forms a PLA. 
This construction defines a functor from the category of Lie skew braces \( \mathbf{LSB} \) to the category of post-Lie algebras \( \mathbf{PLA} \):
\[
D: \mathbf{LSB} \to \mathbf{PLA},
\]
which acts on objects and morphisms as follows:
\[
\begin{aligned}
D(G, \cdot, \circ) &= \bigl(\mathfrak{g}, [\cdot,\cdot]_{\cdot}, [\cdot,\cdot]_{\circ},  \triangleright\bigr), \\
D(\varphi) &= \varphi_{*e},
\end{aligned}
\]
for any homomorphism \( \varphi: (G_1, \cdot_1, \circ_1) \to (G_2, \cdot_2, \circ_2) \) of LSBs,
where the post-Lie product $\triangleright$ is defined in  \eqref{trianglelambda}.

\begin{definition}
A post-Lie algebra \( (\mathfrak{g}, [\cdot,\cdot]_{\cdot}, [\cdot,\cdot]_{\circ},  \triangleright) \) is  \emph{integrable} if there exists {an LSB} \( (G, \cdot, \circ) \) such that
\[
D(G, \cdot, \circ) = (\mathfrak{g}, [\cdot,\cdot]_{\cdot}, [\cdot,\cdot]_{\circ}, \triangleright).
\]
\end{definition}
{
\begin{remark}\label{rem:connected-integrability}
Let \( (\mathfrak{g}, [\cdot,\cdot]_{\cdot}, [\cdot,\cdot]_{\circ},  \triangleright )\)  be an integrable post--Lie algebra (PLA).
Then \( (\mathfrak{g}, [\cdot,\cdot]_{\cdot}, [\cdot,\cdot]_{\circ},  \triangleright )\) is integrable  by a
\emph{connected} LSB. Indeed, if \(D(G, \cdot, \circ) = (\mathfrak{g}, [\cdot,\cdot]_{\cdot}, [\cdot,\cdot]_{\circ}, \triangleright)\) for some (possibly disconnected)
LSB $(G,\cdot,\circ)$, let $G_0$ be the identity component of $G$
(as a manifold). Since the two group structures share the same underlying
manifold, $G_0$ is a Lie subgroup of both $(G,\cdot)$ and $(G,\circ)$, and the
$\lambda$-maps preserve $G_0$. Hence $(G_0,\cdot,\circ)$ is a connected LSB, and since $T_eG_0=T_eG$ we have
\[
D(G_0,\cdot,\circ)=D(G,\cdot,\circ)= (\mathfrak{g}, [\cdot,\cdot]_{\cdot}, [\cdot,\cdot]_{\circ}, \triangleright).
\]
\end{remark}}

As already mentioned in the Introduction, the problem of determining which post-Lie algebras are integrable  is highly nontrivial and remains open, as originally posed in \cite{Bai2023}. A simple example of non integrable PLA is the following.

\begin{example}\label{exnonint}
Consider the pre-Lie algebra \( \left(\mathbf{M}_n(\mathbb{R}), [\cdot,\cdot]_{\cdot} = 0, [\cdot,\cdot]_{\circ},\ \triangleright \right) \), where \( [X,Y]_{\circ} := XY - YX \) is the usual matrix commutator, and the post-Lie product is defined by
\[
X \triangleright Y := XY,
\]
i.e., matrix multiplication. In this case, the Lie algebra \( (\mathbf{M}_n(\mathbb{R}), [\cdot,\cdot]_{\cdot} = 0) \) is abelian, while \( (\mathbf{M}_n(\mathbb{R}), [\cdot,\cdot]_{\circ}) \) is reductive (since it is isomorphic to \( \mathfrak{gl}_n(\mathbb{R}) \)) and not solvable for \( n \geq 2 \). 
Therefore, this PLA does not satisfy condition \textbf{(S2)} in Theorem~\ref{mainteor},.
Hence, it is not integrable.

Now consider the case \( n = 1 \). In this situation, \( \mathbf{M}_1(\mathbb{R}) \cong \mathbb{R} \), and both Lie brackets vanish:
\[
[x,y]_{\cdot} = 0, \qquad [x,y]_{\circ} = xy - yx = 0,
\]
so the underlying Lie algebras are abelian. The post-Lie product becomes \( x \triangleright y = xy \), the standard multiplication in \( \mathbb{R} \). Suppose, for contradiction, that this structure integrates to {an LSB} which is necessarily 
 \( (\mathbb{R}, +, +) \).
For such a trivial brace, we have
\[
\lambda_a(b) = -a + (a + b) = b, \qquad \text{for all } a, b \in \mathbb{R},
\]
so the \( \lambda \)-map is the identity on \( \mathbb{R} \), i.e., \( \lambda(\mathbb{R}) = \mathrm{id}_{\mathbb{R}} \in \mathrm{Aut}(\mathbb{R}) \), and hence its differential at the identity is the zero map:
$\lambda_{*1} = 0$. By  \eqref{trianglelambda}, this implies
\[
x \triangleright y = \lambda_{*1}(x)(y) = 0(y) = 0,
\]
which contradicts  \( x \triangleright y = xy \). 
Thus, even in the case \( n = 1 \), the PLA is not integrable.
\end{example}

\noindent
{\bf Convention.}
In what follows, and with a slight abuse of notation, we shall denote a post-Lie algebra simply by the triple
\( (\mathfrak{g}, [\cdot,\cdot]_{\cdot}, [\cdot,\cdot]_{\circ})\)
omitting the post-Lie product~\(\triangleright\), which is understood to be fixed throughout.

\subsection{Affine actions}\label{affineactions}

We begin by recalling the notion of a simply transitive affine action. 
\begin{definition}
Let \( (G, \circ) \) and \( (K, \cdot) \) be two connected real Lie groups. An \emph{affine action} of \( G \) on \( K \) is a Lie group homomorphism
\[
\rho: (G, \circ) \longrightarrow \mathrm{Aff}(K, \cdot) := K \rtimes \mathrm{Aut}(K, \cdot).
\]
We say that the affine  action is \emph{simply transitive} if  for all \( k_1, k_2 \in K \), there exists a unique \( g \in G \) such that \( \rho(g)(k_1) = k_2 \).
\end{definition}

In the 1970s, Milnor  \cite{Milnor1977} posed a fundamental question:
\begin{quote}
\emph{Does every solvable, connected and simply-connected Lie group \((G, \circ)\) of dimension $n$ admit a simply transitive affine action on \((\mathbb{R}^n, +)\)?}
\end{quote}

This question stimulated extensive research and initially led to several attempts to give a positive answer. However, it was eventually shown that the answer is negative. Specifically, Benoist proved in \cite{Benoist1995nil} the existence of an 11-dimensional solvable (even nilpotent), connected and simply connected Lie group \((G, \circ)\) that does \emph{not} admit a simply transitive affine action on \((\mathbb{R}^{11}, +)\).
Later,  Burde and Grunewald \cite{Burde1995} generalize this to a family of examples in dimension $11$ in  and also exhibited a new example in dimension $10$. 
The core of their argument is to exhibit Lie groups $(G,\circ)$
whose Lie algebra $(\mathfrak g,[\cdot,\cdot]_\circ)$ admits \emph{no}
post-Lie product $\triangleright$ capable of endowing $\mathfrak g$ with a
pre-Lie structure $(\mathfrak g,[\cdot,\cdot]_\circ,[\cdot,\cdot]_{\!\cdot}=0)$.

On the other hand, if one does not insist on the affine action being on an abelian Lie group, a positive result can still be obtained:

\begin{theorem}[Dekimpe \cite{Dekimpe1999semisimple}]\label{dekimpesolvable}
Given a connected and simply-connected solvable Lie group \( G \), there exists a nilpotent Lie group \( K \) such that \( G \) acts simply transitively by affine transformations on \( K \).
\end{theorem}

The following proposition illustrates a natural connection between LSB structures on a Lie group and simply transitive affine actions. While this connection is not always stated explicitly, it is implicit in several works (see, e.g. \cite{Burde2009}-\cite{Burde2024postliealgebraperfect}), especially at the level of post-Lie algebras and in the case where $(K, \cdot_K)$ is simply-connected and nilpotent (in this case  one speaks of   {\em simply transitive Nill-affine actions}). This proposition makes the relationship precise for our purposes in the general setting, namely for general connected real Lie groups.

\begin{theorem}\label{propLSBaff}
Let \( (G, \circ_G) \) be a connected real Lie group. The following are equivalent:
\begin{enumerate}
    \item There exists a Lie group structure \( \cdot_G \) on \( G \) such that \( (G, \cdot_G, \circ_G) \) is {an LSB};
    \item There exists a connected Lie group \( (K, \cdot_K) \) and a simply transitive affine action
    \[
    \rho: (G, \circ_G) \to \mathrm{Aff}(K, \cdot_K).
    \]
\end{enumerate}
Moreover, whenever one (hence both) of the above conditions holds, there is a \emph{canonical} LSB structure $(K,\cdot_K,\circ_K)$ on $K$ and an isomorphism of LSBs 
\[
\Phi: (G, \cdot_G, \circ_G) \xrightarrow{\;\cong\;} (K, \cdot_K, \circ_K).
\]
As a consequence, the corresponding post-Lie algebras
$\bigl(\mathfrak{g}, [\cdot,\cdot]_{\cdot_G}, [\cdot,\cdot]_{\circ_G} \bigr)$ and 
$\bigl(\mathfrak{k}, [\cdot,\cdot]_{\cdot_K}, [\cdot,\cdot]_{\circ_K} \bigr)$
are isomorphic.
\end{theorem}

\begin{proof}
\noindent
(1) \( \Rightarrow \) (2): Suppose \( (G, \cdot_G, \circ_G) \) is {an LSB}. Then the map
\[
\rho: (G, \circ_G) \to \mathrm{Aff}(G, \cdot_G), \quad a \mapsto (a, \lambda_a),
\]
is a Lie group homomorphism, where \( \lambda_a \in \mathrm{Aut}(G, \cdot_G) \) is defined as in \eqref{lambdaspec}. The induced action
\[
\rho(a)(b) := a \cdot_G \lambda_a(b)
\]
is simply transitive on \( (G, \cdot_G) \). We can then set \( (K, \cdot_K) := (G, \cdot_G) \), and \( \rho \) defines a simply transitive affine action.

\medskip

\noindent
(2) \( \Rightarrow \) (1): Assume there exists a simply transitive affine action
\[
\rho: (G, \circ_G) \to \mathrm{Aff}(K, \cdot_K).
\]
We first show that \( \rho(G) \subseteq \mathrm{Aff}(K, \cdot_K) \) is a regular subgroup (cf. Definition  \ref{regsubgroup}).  
Since the action of \( G \) on \( K \) via \( \rho \) is simply transitive, for any \( k \in K \), there exists a unique \( g \in G \) such that
\[
\rho(g)(k) = x \cdot f(k) = e_K,
\]
where $\rho (g)=(x, f)$ and $e_K$ denotes the identity element of $(K, \cdot_K)$.
 It follows that  \( \rho(G) \) is a regular subgroup of \( \mathrm{Aff}(K, \cdot_K) \).
 {Since $\rho$ is an injective homomorphism of Lie groups it follows by \cite[Prop. 7.17]{Lee-ISM} 
 and by Remark \ref{rmkclosed} that $\rho (G)$
 is a closed subgroup of $\mathrm{Aff}(K, \cdot_K)$.}

Now, by Proposition~\ref{proplambda}, every regular Lie  subgroup of \( \mathrm{Aff}(K, \cdot_K) \) corresponds to {an LSB} structure \( (K, \cdot_K, \circ_K) \), and the group \( \rho(G) \) is isomorphic to \( (K, \circ_K) \).
Let \( \Psi: \rho(G) \to (K, \circ_K) \) be such an isomorphism. Then the composition
\[
\Phi := \Psi\circ  \rho: (G, \circ_G) \to (K, \circ_K)
\]
is a Lie group isomorphism. We define a new multiplication on \( G \) by
\[
a \cdot_G b := \Phi^{-1}(\Phi(a) \cdot_K \Phi(b)),
\]
which turns \( G \) into a Lie group and makes \( (G, \cdot_G, \circ_G) \) {an LSB} isomorphic to \( (K, \cdot_K, \circ_K) \). The result follows.
\end{proof}
\medskip
\noindent

Theorem \ref{propLSBaff} shows that studying LSB structures on Lie groups is equivalent to studying simply transitive affine actions. Consequently, the structural properties of LSBs can be translated into geometric properties of affine actions, and viceversa.

The following corollary summarizes several consequences of our results in the setting of affine actions. It provides a reformulation of Theorem~\ref{mainteor} and statement \textbf{(F1)} in Theorem~\ref{mainteor2} from the perspective of affine geometry.

\begin{cor}
Let $\rho: (G,\circ_G) \;\longrightarrow\; \mathrm{Aff}(K,\cdot_K)$
be a simply transitive affine action of a connected Lie group \((G,\circ)\) on a Lie group \((K,\cdot)\). Then:

\begin{itemize}
  \item[\textbf{(S1')}] If \((K,\cdot_K)\) is linear and simply-connected, then \((G,\circ_G)\) is linear.
  \item[\textbf{(S2')}] If \((K,\cdot_K)\) is solvable, then \((G,\circ_G)\) is solvable.
  \item[\textbf{(R1')}] If \((G,\circ_G)\) is nilpotent, then \((K,\cdot_K)\) is solvable.
  \item[\textbf{(R2')}] If \((G,\circ_G)\) is semisimple, then \( \rho \) induces a group isomorphism:
  $(G,\circ_G)\;\cong\;(K,\cdot_K).$

  Moreover, 
  \item[\textbf{(F1')}] For every connected and simply-connected non-linear Lie group \((K,\cdot_K)\), there exists a real connected  linear {Lie group}   \((G,\circ)\) and a  simply transitive affine action
 $\rho: (G,\circ_G) \longrightarrow \mathrm{Aff}(K,\cdot_K)$.
 \end{itemize}
\end{cor}

\medskip

\noindent
Notice that the proof of \textbf{(F2)} in Theorem~\ref{mainteor2} follows by combining Theorem~\ref{propLSBaff} with Dekimpe’s Theorem~\ref{dekimpesolvable}.

Moreover, the result by Benoist mentioned above can be reformulated as the existence of connected and simply connected nilpotent Lie groups of dimension~$11$, denoted by \( (G, \circ) \), such that for any LSB structure \( (G, \cdot, \circ) \), the Lie group \( (G, \cdot) \) is non-abelian.

\begin{remark}\rm
It is also worth noting that a similar phenomenon occurs in the finite setting. Indeed, 
{Bachiller~\cite{Bachiller2016Counterexample}} proved the existence of a finite solvable group \( (G, \circ) \) such that for every left skew brace \( (G, \cdot, \circ) \), the multiplicative group \( (G, \cdot) \) is non-abelian.
\end{remark}

The key idea underlying Benoist's construction relies on the final part of Theorem~\ref{mainteor3}, combined with the observation that certain pairs of Lie algebras, namely \( (\mathbb{R}^{11}, [\cdot,\cdot]_{\circ}) \) and the abelian Lie algebra \( (\mathbb{R}^{11}, [\cdot,\cdot]_{\cdot} = 0) \), cannot be equipped with a bilinear operation \( \triangleright \) satisfying  the pre-Lie identity~\eqref{idpre}. 
\medskip

\noindent
Further results concerning the transition from affine structures to LSB structures will be employed in the proof of Theorem~\ref{mainteor3}.

\section{Proofs of Theorems \ref{mainteor}-\ref{mainteor3} and Corollary \ref{maincor}}\label{proofs}
The following lemmata will be used in the proofs of properties (S1) and (S2) in Theorem~\ref{mainteor}, respectively.
The first lemma establishes a structural property that, although implicitly used in the literature, is often not stated explicitly. For completeness, we include a full proof here.

\begin{lemma} \label{lemmalinear}
Let $(G, \cdot)$ be a linear connected and simply connected real Lie group. 
Then the identity component \( \mathrm{Aff}_{0}(G, \cdot) \) of the group of affine transformations \( \mathrm{Aff}(G, \cdot) \) is also linear.
\end{lemma}

\begin{proof}
By definition, the identity component of the group of affine transformations is given by
\[
\mathrm{Aff}_{0}(G, \cdot) = (G, \cdot) \rtimes \mathrm{Aut}_{0}(G, \cdot),
\]
where \( \mathrm{Aut}_{0}(G, \cdot) \) denotes the connected component of the identity in the group of Lie group automorphisms of \( G \).
Since \( G \) is simply-connected, by {\cite[Thm.~3.18.13]{Varadarajan}} it admits a Levi-Malcev decomposition
$(G, \cdot) = R \rtimes S,$
where \( R \) is the radical of \( G \) and \( S \) is a semisimple subgroup. Thus, in the terminology in \cite{Wu1986}, the subgroup \( R \) is a {\em nucleus}, i.e. a closed  connected and simply-connected normal, solvable Lie subgroup of $G$
such that $G/R$ is reductive.
 Since $R$ is   invariant under the action of \( \mathrm{Aut}_{0}(G) \), it follows  by  \cite[Lemma 2]{Wu1986} that 
 there exists a  faithful representation of \( G \)  which can be extended to a (possibly non-faithful) representation of the semidirect product \( (G, \cdot) \rtimes \mathrm{Aut}_{0}(G, \cdot) \).
Now consider the short exact sequence
\[
1 \rightarrow (G, \cdot) \rightarrow (G, \cdot) \rtimes \mathrm{Aut}_{0}(G) \rightarrow \mathrm{Aut}_{0}(G) \rightarrow 1.
\]
Since \( (G, \cdot) \) admits a faithful linear representation that extends to the semidirect product, and \( \mathrm{Aut}_{0}(G, \cdot) \) is linear, being  a Lie subgroup of \( \mathrm{GL}(\mathfrak{g}) \), where \( \mathfrak{g} = \mathrm{Lie}(G, \cdot) \), we may apply \cite[Lemma 1]{Wu1986} to conclude that \( \mathrm{Aff}_{0}(G, \cdot) \) is linear.
\end{proof}


Our second lemma shows how {an LSB} structure can be lifted to the universal covering of a connected LSB.

\begin{lemma} \label{Universalcover}
Let $(G, \cdot, \circ)$ be a connected LSB with identity element $e \in G$. Let $\widetilde{G}$ be the universal covering space of $G$, and let $\pi: \widetilde{G} \rightarrow G$ denote the covering map. Then, for any $\tilde{e} \in \pi^{-1}(e)$, there exists {an LSB} structure $(\widetilde{G}, \tilde{\cdot}, \tilde{\circ})$ with identity $\tilde{e}$ such that $\pi$ is {an LSB} homomorphism from $(\widetilde{G}, \tilde{\cdot}, \tilde{\circ})$ to $(G, \cdot, \circ)$.
\end{lemma}

\begin{proof}
By standard results in Lie group theory (see, e.g., \cite[Section 3.24]{Warner}), there exist two Lie group structures on $\widetilde{G}$, denoted by $(\widetilde{G}, \cdot)$ and $(\widetilde{G}, \circ )$, both having $\tilde{e}$ as identity element, such that $\pi$ is a Lie group homomorphism from $(\widetilde{G}, \cdot )$ to $(G, \cdot)$ and from $(\widetilde{G}, \circ )$ to $(G, \circ)$ (here with a slight abuse of notation we are denoting the operations on $\widetilde{G}$ with the same symbol $\cdot$ and $\circ$ used for $G$).

To show that $(\tilde G, \cdot , \circ)$ is {an LSB} and that $\pi$ is {an LSB} homomorphism, we must verify that the identity
\eqref{eqLSB} is satisfied, namely
\begin{equation}\label{tildeLSB}
[ \tilde{a} \circ (\tilde{b} \cdot \tilde{c}) ] \cdot  [(\tilde{a} \circ \tilde{b}) \cdot  \tilde{a}^{-1}  \cdot  (\tilde{a} \circ  \tilde{c}) ]^{-1}=\tilde e
\end{equation}holds for all $\tilde{a}, \tilde{b}, \tilde{c} \in \widetilde{G}$.

Define the smooth map
\[
\tilde{f} : \widetilde{G} \times \widetilde{G} \times \widetilde{G} \rightarrow \widetilde{G}, \quad
(\tilde{a}, \tilde{b}, \tilde{c}) \mapsto [ \tilde{a} \circ (\tilde{b} \cdot \tilde{c}) ] \cdot  [(\tilde{a} \circ \tilde{b}) \cdot  \tilde{a}^{-1}  \cdot  (\tilde{a} \circ  \tilde{c}) ]^{-1},
\]
and consider the constant map
\[
c_{\tilde{e}} : \widetilde{G} \times \widetilde{G} \times \widetilde{G} \rightarrow \widetilde{G}, \quad
(\tilde{a}, \tilde{b}, \tilde{c}) \mapsto \tilde{e}.
\]

We aim to prove that $\tilde{f} = c_{\tilde{e}}$.
In order to show this,
notice first  that  $\tilde{f}$ and ${c}_{\tilde e}$ are smooth lifts of the constant map 
$$c_e:G\times G\times G\rightarrow G, \ \ (a, b, c)\mapsto c_e(a, b, c)=e,$$
i.e. $\pi\tilde f=\pi c_{\tilde e}=e$.
Indeed, for $c_{\tilde e}$ {it} is obvious while for $\tilde f$:
$$\pi\tilde f(\tilde a, \tilde b, \tilde c)=\pi\left([ \tilde{a} \circ (\tilde{b} \cdot \tilde{c}) ] \cdot  [(\tilde{a} \circ \tilde{b}) \cdot  \tilde{a}^{-1}  \cdot  (\tilde{a} \circ  \tilde{c}) ]^{-1}\right)=[ a \circ (b \cdot c) ] \cdot  [(a \circ b) \cdot  a^{-1}  \cdot  (a \circ  c) ]^{-1}=e,$$
with $a=\pi(\tilde a), b=\pi(\tilde b)$ and $c=\pi(\tilde c)$, where the second equality follows 
by the fact that $\pi$ is an homomorphism
from $(\widetilde{G}, \cdot )$ to $(G, \cdot)$ and from $(\widetilde{G}, \circ )$ to $(G, \circ)$ 
and the third equality follows by the fact that $(G, \cdot, \circ)$ is {an LSB} and so \eqref{eqLSB} holds true.
Moreover, $\tilde f$ and $c_{\tilde e}$ agree at the point $(\tilde{e}, \tilde{e}, \tilde{e})$, i.e. 
{$c_{\tilde e}(\tilde e, \tilde e, \tilde e)=\tilde f (\tilde e, \tilde e, \tilde e)=\tilde e$}.
By uniqueness of  lifts on covering spaces (since the space involved are connected), it follows that $\tilde f = c_{\tilde e}$, and we are done.  
\end{proof}

\begin{remark}\rm
Thanks to the correspondence between LSBs and simply transitive affine actions established in Theorem~\ref{propLSBaff}, Lemma~\ref{Universalcover} admits the following geometric counterpart:
\emph{Let
$\rho : (G,\circ)\;\longrightarrow\;\mathrm{Aff}(K,\cdot)$
be a simply transitive affine action of a connected Lie group $(G,\circ)$ on another connected Lie group $(K,\cdot)$.  
Then the universal coverings $\widetilde{G}$ and $\widetilde{K}$ carry a lifted action
$\widetilde{\rho} : (\widetilde{G},\circ)\;\longrightarrow\;\mathrm{Aff}(\widetilde{K},\cdot),$
which is again simply transitive.}
\end{remark}

\begin{proof}[Proof of Theorem \ref{mainteor}]
\mbox{}

\noindent
\textbf{(S1)} 
Let $(G, \cdot, \circ)$ be {an LSB}, with $(G, \cdot )$ connected simply-connected and linear. By Proposition~\ref{proplambda}, there exists a regular subgroup \( H \leq \mathrm{Aff}(G, \cdot) \) such that \( H \cong (G, \circ) \).
Since \( G \) is connected, it follows that \( H \leq \mathrm{Aff}_0(G, \cdot) \). 
As \( \mathrm{Aff}_0(G, \cdot) \) is linear (by Lemma~\ref{lemmalinear}),  \( H \) is also linear. Therefore, \( (G, \circ) \) is linear as well.

\vspace{0.5em}
\noindent
\textbf{(S2)} Let \( (G, \cdot, \circ) \) be a connected LSB such that \( (G, \cdot) \) is solvable. We aim to prove that \( (G, \circ) \) is also solvable.
By Lemma~\ref{Universalcover}, we may assume without loss of generality that \( G \) is simply connected. That is, \( (G, \cdot) \) is a connected, simply connected, solvable Lie group.
By standard results, such a Lie group is linear (see \cite[Th. 16.2.7]{Hilgert1952book}) and diffeomorphic to \( \mathbb{R}^n \) (see, e.g., \cite[Th. $2^a$]{Chevalley1951}). Hence, the Lie group \( (G, \circ) \), which is defined on the same underlying manifold \( G \), is also diffeomorphic to \( \mathbb{R}^n \).

By the Levi–Malcev decomposition, we have
$(G, \circ) \simeq R \rtimes S,$
where \( R \) is a connected, simply connected solvable Lie group, and \( S \) is a connected, simply connected semisimple Lie group.

Since \( (G, \circ) \) is diffeomorphic to \( \mathbb{R}^n \), it is in particular contractible. This implies that the semisimple part \( S \) must also be contractible, and hence diffeomorphic to some Euclidean space. One can verify  (see, e.g., \cite[Ch. VI]{Knapp2013lie}) that the \emph{only} connected, simply connected, non-compact, simple real Lie group diffeomorphic to an Euclidean space is the universal covering group \( \widetilde{SL_2(\mathbb{R})} \) of \( SL_2(\mathbb{R}) \) the special linear group of order $2$.
Since any semisimple Lie group is a finite direct product of simple Lie groups, we conclude that
$S \simeq \widetilde{SL_2(\mathbb{R})}^k$
for some \( k \geq 0 \). Therefore, 
$(G, \circ) \simeq R \rtimes \widetilde{SL_2(\mathbb{R})}^k,$
where \( R \) is a connected, simply connected solvable Lie group.
By \textbf{(S1)}, the Lie group \( (G, \circ) \) is linear. However, it is well known that \( \widetilde{SL_2(\mathbb{R})} \) is not linear. This forces \( k = 0 \). Hence,
$(G, \circ) \simeq R,$
and since \( R \) is solvable, we conclude that \( (G, \circ) \) is solvable.

\vspace{0.5em}
\noindent
\textbf{(R1)} 
Let \( (G, \cdot, \circ) \) be a connected LSB such that \( (G, \circ) \) is nilpotent 
and let \( (\mathfrak{g}, [\cdot,\cdot]_{\cdot}, [\cdot,\cdot]_{\circ})\) be the corresponding PLA . Then the Lie algebra \( (\mathfrak{g}, [\cdot, \cdot]_{\circ}) \) is nilpotent as well. By \cite[Prop.~3.2]{Burde2013postliealgebra}, this implies that the Lie algebra \( (\mathfrak{g}, [\cdot, \cdot]_{\cdot}) \) is solvable. Hence, \( (G, \cdot) \) is solvable, as desired.

\vspace{0.5em}
\noindent
\textbf{(R2)} 
Let \( (G, \cdot, \circ) \) be a connected Lie skew brace such that \( (G, \circ) \) is semisimple 
and let \( (\mathfrak{g}, [\cdot,\cdot]_{\cdot}, [\cdot,\cdot]_{\circ})\) be the corresponding PLA. Then the Lie algebra \( (\mathfrak{g}, [\cdot, \cdot]_{\circ}) \) is semisimple as well.
By adapting the proof of \cite[Th.~3.3]{Burde2022rigiditypostliealgebra} to the real setting,\footnote{Although \cite[Th.~3.3]{Burde2022rigiditypostliealgebra} is stated over \(\mathbb{C}\), the same rigidity argument can be transferred to the real case by passing to the complexifications and then descending to the real forms.} we conclude that \( (\mathfrak{g},[\cdot,\cdot]_{\cdot}) \cong (\mathfrak{g},[\cdot,\cdot]_{\circ}) \) as Lie algebras. Therefore, \( (G,\cdot) \) and \( (G,\circ) \) are locally isomorphic Lie groups.
\end{proof}

\begin{remark}\rm\label{rmk1mainteor}
Point \textbf{(S1)} of Theorem~\ref{mainteor} \emph{fails} once the assumption that $G$ is simply connected is dropped. An explicit counterexample can be constructed as follows.
Let \( H = \mathbb{R}^3 \) equipped with the Heisenberg group law:
\begin{equation}\label{heislaw}
(x,y,z)\cdot(x',y',z') =
\bigl(x + x',\, y + y',\, z + z' + xy' \bigr).
\end{equation}

Consider the discrete central subgroup
\[
N = \{(0,0, k): k \in \mathbb{Z} \} \subset Z(H),
\]
and define \( G := H / N \), denoting the group operation on \( G \) by \( \circ \). We represent cosets in \( G \) as \([x, y, \theta]\), with \(\theta \in \mathbb{S}^1\).
As shown in \cite[Section 4.8]{Hall2003lie}, the group \( G \) is not linear. Now consider the following subgroups of \( G \):
\[
G_1 = \bigl\{ [x, 0, \theta] : x \in \mathbb{R},\, \theta \in \mathbb{S}^1 \bigr\}
\cong \mathbb{R} \times \mathbb{S}^1, \quad
G_2 = \bigl\{ [0, y, 0] : y \in \mathbb{R} \bigr\} \cong \mathbb{R}.
\]

It is readily verified that \( G_1 \) is a closed normal subgroup of \( G \), \( G_2 \) is a closed subgroup, and \( G_1 \cap G_2 = \{e\} \). Therefore, \( G \) is the semidirect product \( G = G_1 \rtimes G_2 \). 
Now endow the same underlying set
$G = (\mathbb{R} \times \mathbb{S}^1) \times \mathbb{R}$
with the direct product operation $\cdot$. According to Example~\ref{exLSBrevzappa}, the pair of operations \((\cdot, \circ)\) defines a connected but non-simply-connected LSB, where the group \((G, \cdot)\) is linear, although \((G, \circ )\) is not.
This shows that the assumption of simple-connectedness is crucial for the validity of point \textbf{(S1)}.
\end{remark}

\medskip

\begin{proof}[Proof of Corollary~\ref{maincor}]
Let \(G_{0}\subseteq G\) be the (closed) connected component of the identity.  
Because the map \(\lambda_{a}\) defined in~\eqref{lambdaspec} is continuous, we have
\(\lambda_{a}(G_{0})=G_{0}\) for every \(a\in G\). Moreover,  \(G_{0}\) is a normal
subgroup of both \((G,\cdot)\) and \((G,\circ)\).

Consequently, we obtain two compact LSBs:
\[
  (G_{0},\cdot,\circ)
  \quad\text{and}\quad
  (\overline{G},\cdot,\circ),\qquad
  \overline{G}=G/G_{0},
\]
where the quotient operations are
\[
  (aG_{0})\cdot(bG_{0})=(a\cdot b)G_{0},
  \qquad
  (aG_{0})\circ(bG_{0})=(a\circ b)G_{0}.
\]

Because \(G_{0}\) is connected, condition \textbf{(S2)} of
Theorem~\ref{mainteor} tells us that the  group \((G_{0},\circ)\) is
solvable.
Next, \((G,\cdot)\) is compact and nilpotent, and \(G_{0}\) is an open
subgroup, so \(G_{0}\) has finite index in \((G,\cdot)\).  Therefore
\(|\overline{G}|<\infty\), and \((\overline{G},\cdot,\circ)\) is a finite left
skew brace.  Nilpotency is preserved by quotients, hence
\((\overline{G},\cdot)\) is finite and nilpotent.  By
\cite[Cor.~2.23]{Smoktunowicz2018} the multiplicative group
\((\overline{G},\circ)\) is solvable.
Since the class of solvable groups is closed under extensions, the 
group \((G,\circ)\) is also solvable.  This completes the proof
\footnote{Motivated by the discrete left skew brace setting
(cf.\ \cite{KonSmokVen2021}), one may call a subset
\(I\subseteq(G,\cdot,\circ)\) an \emph{ideal} if it is a closed Lie subgroup
satisfying
\begin{enumerate}\renewcommand{\labelenumi}{(\roman{enumi})}
  \item \(I\trianglelefteq(G,\cdot)\),
  \item \(I\trianglelefteq(G,\circ)\),
  \item \(\lambda_{a}(I)=I\) for every \(a\in G\).
\end{enumerate}
The argument above shows that \(G_{0}\) is indeed such an ideal.  The
closedness of \(I\) is precisely what guarantees that the quotient
\(G/G_{0}\) carries a natural manifold structure.}.
\end{proof}

\begin{remark}\label{rmkconj}
It is worth pointing out that, if Smoktunowicz--Vandramin's conjecture stated in the Introduction is true, the nilpotency assumption in Corollary~\ref{maincor} can be relaxed to mere solvability.
In any case, repeating the argument above and exploiting the fact that, for every \emph{finite} left skew brace \((B,\cdot,\circ)\), the solvability of \((B,\cdot)\) together with \(|B|\) not divisible by~\(3\) forces \((B,\circ)\) to be solvable 
\cite[Cor. 2.3]{Gorshkov2021}, we obtain the following consequence: {\em 
let \((G,\cdot,\circ)\) be a  LSB such that \((G,\cdot)\) is solvable and the quotient \(G/G_0\) is finite of order not divisible by~\(3\).
Then the group \((G,\circ)\) is solvable.}
\end{remark}

\vspace{0.5em}

\begin{proof}[Proof of Theorem \ref{mainteor2}]
\mbox{}

\noindent
\textbf{(F1)} 
Let \( (G, \cdot) \) be a connected, simply connected, non-linear Lie group.  
Our objective is to endow \( G \) with a second Lie group structure \( \circ \) such that \( (G, \cdot, \circ) \) forms {an LSB}, and the group \( (G, \circ) \) is linear.

To this end, we aim to show that there exist two linear Lie subgroups \( G_1, G_2 \subseteq G \) satisfying
\begin{equation}\label{scomp}
   G = G_1 G_2, \qquad G_1 \cap G_2 = \{e\}.
\end{equation}
Such a factorization allows us, via Example~\ref{exLSBzappa}, to define a second group operation \( \circ \) on \( G \) such that
\[
   (G, \circ) \cong G_1 \times G_2,
\]
where the right-hand side is endowed with the direct product structure, which is linear. 

In order to get \eqref{scomp} we use again the 
 Levi--Malcev decomposition. So let \( R \trianglelefteq G \) be a connected and simply connected solvable Lie subgroup,  and \( S \le G \) be  a connected and simply connected semisimple Lie {group}, such that
\[
G = RS, \qquad R \cap S = \{e\}.
\]
Moreover, by the \emph{Iwasawa decomposition} (see \cite[Th.~6.46]{Knapp2013lie}), we can write
\[
S = KAN,
\]
where \( AN \) is a connected and simply connected solvable Lie subgroup, and \( K = U \times V \), with \( U \cong \mathbb{R}^n \), \( V \) compact, connected, and simply connected. Furthermore, \( K \cap AN = \{e\} \).

Combining these, we obtain:
\[
G = RS = R(KAN) = (RK)(AN) = (RUV)(AN).
\]
Define
\[
G_1 := RUV, \qquad G_2 := AN.
\]
It remains to show that \( G_1 \cap G_2 = \{e\} \), and that \( G_1 \) and \( G_2 \) are linear Lie subgroups of \( G \). Let \( x \in G_1 \cap G_2 \). Then \( x = ruv \) for some \( r \in R \), \( u \in U \), \( v \in V \), and also \( x = an \) for some \( a \in A \), \( n \in N \). Hence,
\[
ruv = an \quad \Rightarrow \quad r = an (uv)^{-1} \in ANK.
\]
Since \(  KAN=S \le G \), we have \( ANK = KAN \), so \( r \in KAN \). But \( R \cap KAN = \{e\} \), thus \( r = e \), implying
$uv = an \in AN$.
However, \( uv \in K \), and \( K \cap AN = \{e\} \), hence \( uv = e \), and so \( x = e \). Therefore, \( G_1 \cap G_2 = \{e\} \). Next, we prove that both \( G_1 \) and \( G_2 \) are linear. Note that
$G_1 = RUV = (RU)V.$
Since \( R \trianglelefteq G \) and \( R \cap U = \{e\} \), we have \( RU \cong R \rtimes U \), which is connected, simply connected, and solvable. Also, \( V \) is compact and therefore real linearly reductive. Since $UV$ normalizes $R$,  
it follows that \( RU \trianglelefteq G_1 \), and hence
$G_1 = (RU) \rtimes V$.
By \cite[Th.~16.2.7]{Hilgert1952book}, such a group is linear. Moreover, \( G_2 = AN \) is connected, simply connected, and solvable, so it is linear by the same result.

\vspace{0.5em}
\noindent
\textbf{(F2)} 
Let \( (G,\circ) \) be a  connected and simply connected solvable Lie group.  
Combining Theorem  \ref{propLSBaff} with Theorem \ref{dekimpesolvable}, one can construct {an LSB} \( (G,\cdot,\circ) \) such that the Lie group \( (G,\cdot) \) is nilpotent.

\end{proof}

\begin{proof}[Proof of Theorem \ref{mainteor3}]
To justify each entry of  the table in Theorem \ref{mainteor}, we examine every pair \((G,\cdot)\) and \((G,\circ)\) in turn, constructing explicit non-trivial LSB structures for the boxes marked \(\cmark\) and proving non-existence or  triviality for those marked with a dash ($-$).

\medskip 

\paragraph{\textbf{Step 1. Positive entries (\(\cmark\)).}}
\mbox{}

\medskip

\noindent
{\bf 1. \emph{$(G,\cdot)$ ab, $(G,\circ)$ nil}}

\noindent
In \cite{Scheuneman1974}, using pre-Lie algebra techniques, it was shown that Milnor's question (see Subsection~\ref{affineactions}) has a positive answer for \(3\)-step nilpotent Lie groups.  
Combining this result with Theorem~\ref{propLSBaff}, we deduce that for every connected, simply connected nilpotent Lie group \((G,\circ)\) of step~\(3\), there exists another Lie group law \( \cdot \) on the same underlying manifold \( G \) such that the triple \((G,\cdot,\circ)\) forms an LSB structure, with \((G,\cdot)\) abelian.

\medskip

\noindent  
{\bf  2. \emph{$(G,\cdot)$ ab, $(G,\circ)$ solv}}

\noindent
Consider the affine group \( (G, \circ) = \mathrm{Aff}(\mathbb{R}, +) \), 
that is, the group of automorphisms of the additive group \( (\mathbb{R}, +) \). 
This group can be realized as the semidirect product
$(\mathbb{R}, +) \rtimes_\alpha (\mathbb{R}, +),$
where both factors are copies of the additive group of real numbers, 
and the action \( \alpha \colon (\mathbb{R}, +) \to \mathrm{Aut}(\mathbb{R}, +) \) is given by
$\alpha(a)(b) = e^a b.$
The group operation \( \circ \) on \( G \cong \mathbb{R}^2 \) is then explicitly defined by
\[
(a, b) \circ (a', b') = \left( a + a',\, b + e^a b' \right),
\]
for all \( a, a', b, b' \in \mathbb{R} \). The  group \( (G, \circ) \) is solvable but not nilpotent.
By Example~\ref{exLSBrevzappa}, we deduce that this structure induces {an LSB} on \( \mathbb{R}^2 \), denoted \( (\mathbb{R}^2, +, \circ) \), where \( + \) is the standard addition on \( \mathbb{R}^2 \) and \( \circ \) is the nontrivial solvable affine product defined above.
Similar examples can be constructed, via pre-Lie algebra techniques, whenever Milnor's question has a positive answer.

\medskip 

\noindent 
{\bf  3. \emph{$(G,\cdot)$ nil, $(G,\circ)$ ab}} 

\noindent
Let \((H_{3}, \cdot)\) be the three–dimensional Heisenberg group with multiplication
\eqref{heislaw}.
It is easily seen that $H_3=\mathbb{R}^2\rtimes_\alpha \mathbb{R}$ 
where 
\[
\alpha \;:\; (\mathbb{R},+) \;\longrightarrow\; \operatorname{Aut}\!\bigl(\mathbb{R}^{2},+\bigr)
          \;=\; \operatorname{GL}_{2}(\mathbb{R}),
\qquad
\alpha(x)\;=\;
\begin{pmatrix}
1 & 0\\[2pt]
x & 1
\end{pmatrix},
\quad x\in\mathbb{R}.
\]
By applying Example \ref{exLSBzappa} we obtain {an LSB}
$(\mathbb{R}^{3},\,\cdot,\,+),$
in which \((\mathbb{R}^{3},\cdot)\)  is {\em nilpotent} and non-abelian,
whereas \((\mathbb{R}^{3},+)\) is {\em abelian}.

\medskip

\noindent
{\bf  4. \emph{$(G,\cdot)$ nil, $(G,\circ)$ nil}}

\noindent
In {\cite[Proposition 4.1]{Burde2009}}, it is proved that for every \( n \leq 5 \), and for any pair of connected and simply connected nilpotent Lie groups \( (G, \circ_G) \) and \( (K, \cdot_K) \) of dimension \( n \), there exists a simply transitive affine action 
\[
\rho: (G, \circ_G) \rightarrow \mathrm{Aff}(K, \cdot_K).
\]
Therefore, by choosing \( (G, \circ_G) \) and \( (K, \cdot_K) \) to be two non-isomorphic connected and simply connected nilpotent Lie groups of dimension \( n \leq 5 \), and applying Theorem~\ref{propLSBaff}, we obtain a nontrivial LSB structure \( (G, \cdot, \circ) \), where both \( (G, \cdot) \) and \( (G, \circ) \) are connected and simply connected {\em nilpotent} non-abelian Lie groups.

\medskip

\noindent
{\bf  5. \emph{$(G,\cdot)$ nil,  $(G,\circ)$ solv}}

\noindent
Let \( (G, \circ) \) be any connected and simply connected solvable Lie group that is not nilpotent. 
Then, by applying statement \textbf{(F2)} of Theorem~\ref{mainteor2}, one can construct {an LSB} \( (G, \cdot, \circ) \) such that the group \( (G, \cdot) \) is nilpotent.  
\medskip

\noindent
{\bf  6. \emph{$(G,\cdot)$ solv, $(G,\circ)$ ab}}

\noindent
Consider the affine group \( (G, \cdot) = \mathrm{Aff}(\mathbb{R}, +) \cong (\mathbb{R}, +) \rtimes_\alpha (\mathbb{R}, +) \), 
i.e., the group of automorphisms of the additive group \( (\mathbb{R}, +) \), as in item 2. above. 
This time, however, we denote the affine group operation by \( \cdot \) to emphasize its role as the primary group structure. 
By applying Example~\ref{exLSBzappa}, we deduce that this setting induces {an LSB}  structure on \( \mathbb{R}^2 \), denoted \( (\mathbb{R}^2, \cdot, +) \), 
where \( (\mathbb{R}^2, \cdot) \) is a solvable Lie group and \( (\mathbb{R}^2, +) \) is abelian.

\medskip

\noindent
{\bf 7.  \emph{$(G,\cdot)$ solv, $(G,\circ)$ nil}}

\noindent
Let \(H_{3}\) be the three–dimensional Heisenberg group with multiplication \eqref{heislaw}.

For every \(t\in\mathbb{R}\) the map
\[
\alpha_{t}:(x,y,z)\longmapsto\bigl(e^{t}x,\,e^{t}y,\,e^{2t}z\bigr)
\]
is an automorphism, yielding a one–parameter homomorphism
\(\alpha:\mathbb{R}\to\operatorname{Aut}(H_{3})\).

With this action we set
$(G,\cdot)=\mathbb{R}\rtimes_{\alpha}H_{3},$
a four-dimensional Lie group that is solvable but not nilpotent.
In contrast, the direct product
$(G,\circ)=\mathbb{R}\times H_{3}$
is nilpotent (of step 2) and non-abelian.

By applying Example \ref{exLSBzappa} we obtain {an LSB}
$(\mathbb{R}^{4},\,\cdot,\,\circ),$
in which \((\mathbb{R}^{4},\cdot)\) is {\em solvable} and non-nilpotent,
whereas \((\mathbb{R}^{4},\circ)\) is {\em nilpotent} and non-abelian.

\medskip

\noindent
{\bf 8. \emph{$(G,\cdot)$ solv,  $(G,\circ)$ solv}}

\noindent
Following \cite[Example~2.19]{Letourmy}, let \(n\ge 3\) and let
$G := UT_n^{+}(\mathbb{R})$
be the Lie group of invertible upper triangular \(n\times n\) real matrices with positive diagonal
entries. Every \(A\in G\) can be written uniquely as
\[
A = N(A) + D(A),
\]
where \(N(A)\) is strictly upper triangular and \(D(A)\) is diagonal. Let \(\circ\) denote the usual
matrix multiplication on \(G\), and define a second operation \(\cdot\) on \(G\) by
\[
A\cdot B := N(A) + BD(A) \qquad (A,B\in G).
\]
It is straightforward to check that \(\cdot\) is a Lie group law on \(G\), and that
\((G,\cdot,\circ)\) is an LSB, where \((G,\circ)\) is the standard solvable (non-nilpotent) upper
triangular Lie group.
Moreover, one can show that the commutator in \((G,\cdot)\) is given by
\[
[A,B]_{\cdot} = \mathrm{I}_n +N(B)\bigl(D(A)-\mathrm{I}_n\bigr)-N(A)\bigl(D(B)-\mathrm{I}_n\bigr).
\]
Thus, one can verify that \((G,\cdot)\) is solvable but not nilpotent.
 Finally, one easily checks that
\((G,\cdot)\) is metabelian: indeed, its commutator subgroup is
\[
\{\,\mathrm{I}_n+N \mid N \text{ is strictly upper triangular}\,\},
\]
which is abelian with respect to \(\cdot\). On the other hand, \((G,\circ)\) is not metabelian for
\(n\ge 3\), since its commutator subgroup is the unitriangular group, which is non-abelian. Hence
\((G,\cdot)\not\simeq (G,\circ)\), and we obtain a non-trivial LSB \((G,\cdot,\circ)\) with both
groups solvable and not nilpotent.

\medskip

\noindent
{\bf 9.  \emph{$(G,\cdot)$ simple,  $(G,\circ)$ solv}} and {\bf \emph{$(G,\cdot)$ ssimple,   $(G,\circ)$ solv}} 

\noindent
Consider the simple Lie group \( \left( \mathrm{SL}_2(\mathbb{R}), \cdot \right) \), where \( \cdot \) denotes the usual matrix multiplication. Its Iwasawa decomposition takes the form
$\mathrm{SL}_2(\mathbb{R}) = KAN,$
where:
\begin{itemize}
  \item \( K = \mathrm{SO}(2) = \left\{ \begin{pmatrix}
      \cos\theta & -\sin\theta \\
      \sin\theta & \cos\theta
    \end{pmatrix} : \theta \in \mathbb{R} \right\} \)
  is a maximal compact subgroup;
  
  \item \( A = \left\{ \begin{pmatrix}
      a & 0 \\
      0 & a^{-1}
    \end{pmatrix} : a > 0 \right\} \)
  is a one-parameter abelian subgroup consisting of diagonal matrices with positive entries;
  
  \item \( N = \left\{ \begin{pmatrix}
      1 & x \\
      0 & 1
    \end{pmatrix} : x \in \mathbb{R} \right\} \)
  is  isomorphic to the additive group \( (\mathbb{R}, +) \).
\end{itemize}
The subgroup \( K \cong \mathbb{S}^1 \) is compact, connected, and abelian, and therefore nilpotent. The product \( AN \) forms a connected, simply-connected, solvable Lie subgroup. Moreover, \( K \cap AN = \{ I_2 \} \), where \( I_2 \) denotes the \( 2 \times 2 \) identity matrix.
Hence, Example~\ref{exLSBzappa} applies, yielding {an LSB} structure
$\left( \mathrm{SL}_2(\mathbb{R}), \cdot, \circ \right),
$
 where
\( \left( \mathrm{SL}_2(\mathbb{R}), \cdot \right) \) is  simple, while
\( \left( \mathrm{SL}_2(\mathbb{R}), \circ \right) \cong K \times AN \) is solvable.
Finally, taking the product of the previous LSB with itself
(cf.\ Example~\ref{exLSBprod}) yields
$\bigl(\mathrm{SL}_{2}(\mathbb{R})\times\mathrm{SL}_{2}(\mathbb{R}){,} 
 \,\cdot,\,\circ\bigr)$,
where
$\bigl(\mathrm{SL}_{2}(\mathbb{R})\times\mathrm{SL}_{2}(\mathbb{R}),\,\cdot\bigr)$
is semisimple
and 
$(\mathrm{SL}_{2}(\mathbb{R})\times\mathrm{SL}_{2}(\mathbb{R}),\,\circ\bigr)$
is solvable.

\medskip

\noindent
{\bf 10.   \emph{$(G,\cdot)$ simple,  $(G,\circ)$ mixed-type}}  and {\bf \emph{$(G,\cdot)$ ssimple,  $(G,\circ)$ mixed-type}}

\item
Consider the simple Lie group \( \bigl( \mathrm{SL}_3(\mathbb{R}), \cdot \bigr) \),
with the usual matrix multiplication.  
Its Iwasawa decomposition is
$\mathrm{SL}_3(\mathbb{R}) \;=\; K\,A\,N,$
where
\begin{itemize}
  \item \(
        K \;=\; \mathrm{SO}(3)\)
               is a maximal compact subgroup (compact, connected, \emph{semisimple}, non-abelian);

  \item \(
        A
        \;=\;
        \Bigl\{
          \operatorname{diag}(a_1,a_2,a_3)
          \,\bigm|\,
          a_i>0,\;
          a_1a_2a_3=1
        \Bigr\}
        \cong (\mathbb{R}_{>0})^2
        \)
        is a two-parameter abelian subgroup of positive diagonal matrices;

  \item \(
        N
       \)
        is a three-dimensional, step-two nilpotent subgroup
        (isomorphic to the  Heisenberg group \(H_{3}\)).
\end{itemize}

The subgroup \( K\cong \mathrm{SO}(3) \) is compact and semisimple,
while the product \( AN \) is connected, simply connected, and solvable.  
Moreover,
\( K\cap AN = \{I_3\} \), where \( I_3 \) denotes the \(3\times3\) identity.

Hence Example \ref{exLSBzappa} applies and yields {an LSB}
structure $\bigl(\mathrm{SL}_3(\mathbb{R}),\,\cdot,\,\circ\bigr),$
such that
\( \bigl(\mathrm{SL}_3(\mathbb{R}),\cdot\bigr) \) is simple,
whereas
$\bigl(\mathrm{SL}_3(\mathbb{R}),\circ\bigr)
\;\cong\;
K \times AN$
is of \emph{mixed type}: it contains both a non-abelian semisimple factor
(\(K\)) and a solvable factor (\(AN\)), so it is neither solvable, nor simple,
nor semisimple. Analogously with the previous example by taking 
the product of  $\bigl(\mathrm{SL}_3(\mathbb{R}),\,\cdot,\,\circ\bigr),$
with itself one gets  {an LSB} such that 
$\bigl(\mathrm{SL}_{3}(\mathbb{R})\times\mathrm{SL}_{3}(\mathbb{R}),\,\cdot\bigr)$
is semisimple
and 
$(\mathrm{SL}_{3}(\mathbb{R})\times\mathrm{SL}_{3}(\mathbb{R}),\,\circ\bigr)$
is of mixed-type.

\medskip 

\noindent
{\bf 11.  \emph{$(G,\cdot)$ mixed-type,  $(G,\circ)$ solv}}

\smallskip

\noindent
Let \(\bigl(\mathrm{SL}_{2}(\mathbb{R}),\cdot,\circ\bigr)\) be the
LSB constructed in~\textbf{(9)}, for which
\(\bigl(\mathrm{SL}_{2}(\mathbb{R}),\cdot\bigr)\) is \emph{simple}
whereas \(\bigl(\mathrm{SL}_{2}(\mathbb{R}),\circ\bigr)\) is
\emph{solvable}.  
Form the direct product with the abelian LSB
\((\mathbb{R},+,+)\) as in Example~\eqref{exLSBprod}.  
The resulting LSB
$\bigl(\,\mathrm{SL}_{2}(\mathbb{R})\times\mathbb{R},\;
        \cdot,\;\circ\,\bigr)$
 is such that 
$\bigl(\mathrm{SL}_{2}(\mathbb{R})\times\mathbb{R},\cdot\bigr)$
is of {\em mixed-type} and 
$\bigl(\mathrm{SL}_{2}(\mathbb{R})\times\mathbb{R},\circ\bigr)$
 is \emph{solvable}.

\medskip 

\noindent
{\bf 12.  \emph{$(G,\cdot)$ mixed-type,  $(G,\circ)$ mixed-type}}

\noindent
Let $(G,\cdot)=\mathbb{R}^{3}\rtimes_{\operatorname{Ad}}\mathrm{SU}(2)$,
where $\operatorname{Ad}\;:\;\mathrm{SU}(2)\;\longrightarrow\;\operatorname{Aut}(\mathbb{R}^{3})$
is the adjoint representation,
\emph{namely}
\[
\operatorname{Ad}_{u}(v)=u\,v\,u^{-1},
\qquad
u\in\mathrm{SU}(2),\;v\in\mathbb{R}^{3}\cong\mathfrak{su}(2)=\mathrm{Lie}(SU(2)).
\]
By applying Example \ref{exLSBzappa} we obtain {an LSB}
$(\mathbb{R}^{3}\times\mathrm{SU}(2), \cdot, \circ)$
where  both $(\mathbb{R}^{3}\times\mathrm{SU}(2), \cdot )$
and $(\mathbb{R}^{3}\times\mathrm{SU}(2), \circ )$ are of {\em mixed-type} and not isomorphic.

\medskip

\noindent
\textbf{Step 2. The trivial cases \(({\cong})\).} 
A connected LSB \(\bigl(G,\cdot,\circ\bigr)\) is \emph{trivial} whenever the two
underlying Lie group structures \(\bigl(G,\cdot\bigr)\) and
\(\bigl(G,\circ\bigr)\) are both abelian.  
Indeed, on a connected manifold any two abelian Lie
group laws are isomorphic.
The only other trivial situations are those in which
\(\bigl(G,\cdot\bigr)\) and \(\bigl(G,\circ\bigr)\) are simultaneously
semisimple (or even simple).  
These follow directly from condition~\(\mathbf{(R2)}\) in
Theorem~\ref{mainteor}.

\medskip
\noindent
\textbf{Step 3. The non-existence cases \(({-})\).}  
Every remaining entry marked by a dash \(({-})\) represents a
\emph{non-existence} phenomenon.  
Such combinations are ruled out by
\(\mathbf{(S2)}\) together with \(\mathbf{(R1)}\) of
Theorem~\ref{mainteor}.

\end{proof}

{\section*{Acknowledgements}
We would like to thank Giovanni Placini and Fabio Zuddas for carefully reading our manuscript and for their valuable comments and suggestions. We also thank the anonymous referee for a thorough reading and for helpful remarks that improved the exposition of the paper. In particular, we are grateful to the referee for suggesting the inclusion of item~8. in the proof of Theorem \ref{mainteor3}.

\end{document}